\documentclass[11pt,a4paper,oneside]{amsart}
\usepackage{amsmath,amssymb}
\usepackage{bigints}
\usepackage{comment}
\newtheorem{theorem}{Theorem}[section]
\newtheorem{lemma}[theorem]{Lemma}

\newtheorem{definition}[theorem]{Definition}
\newtheorem{remark}[theorem]{Remark}

\numberwithin{equation}{section}

\def\XXint#1#2#3{{\setbox0=\hbox{$#1{#2#3}{\int}$}
 \vcenter{\hbox{$#2#3$}}\kern-.5\wd0}}

\title{Musielak-Orlicz-Sobolev embeddings: Necessary and Sufficient Conditions}
\author{Ankur Pandey and Nijjwal Karak}
\thanks{Both the authors gratefully acknowledge support from DST-SERB (SRG/2021/000118). The first author also acknowledges the Council of Scientific and Industrial Research (CSIR) for awarding a Junior Research Fellowship (File Number: 09/1026(17526)/2024-EMR-I). The authors are grateful to the reviewer for all the suggestions and for the idea of the alternate proof of Theorem 1.2}
\begin{document}
\begin{abstract}
 In this paper, we study the necessary and sufficient conditions in the domain for the Musielak-Orlicz-Sobolev embedding of the space $W^{1,\Phi(\cdot,\cdot)}(\Omega),$ where $\Phi(x,t):=t^{p(x)}{(\log(e+t))}^{q(x)}$.
 \end{abstract}
 \maketitle

\section{Introduction}
We assume throughout the paper that $\Omega$ is an open subset of $\mathbb{R}^n$ and the variable exponents $p$ and $q$ are continuous functions defined on $\Omega$ or $\mathbb {R}^n$, satisfying 
\begin{enumerate}
    \item [(p1)] $1\leq p^-:= inf_{x\in\mathbb R^n} p(x)\leq sup_{x\in\mathbb R^n} p(x)=:p^+<\infty$
    
    \item[(q1)] $-\infty<q^-:=inf_{x\in\mathbb R^n} q(x)\leq sup_{x\in\mathbb R^n} q(x)=:q^+<\infty$.
\end{enumerate}  
The following two conditions on $p$ and $q$ will also be used, which in the literature, are known as the log-H\"older continuous and the log-log-H\"older continuous, respectively:
\begin{enumerate}
\item[(p2)] $|p(x)-p(y)|\leq \frac{C}{\log(e+1/|x-y|)}$ whenever $x\in\mathbb R^n$ and $y\in\mathbb R^n$
\item[(q2)] $|q(x)-q(y)|\leq \frac{C}{\log(e+\log(e+1/|x-y|))}$ whenever $x\in\mathbb R^n$ and $y\in\mathbb R^n $.
\end{enumerate}
For the variable exponent Sobolev space $W^{1,p(\cdot)}(\Omega),$ the Sobolev-type (continuous) embedding $W^{1,p(\cdot)}(\Omega)\hookrightarrow L^{p^*(\cdot)}(\Omega)$ was established in \cite{Die04} for bounded domains with local Lipschitz boundary, with the condition $(p2)$ on the exponent $p.$ 

For Musielak-Orlicz-Sobolev spaces, Sobolev-type embeddings have been studied in \cite{CD, Fan12, HMOS10, MOS18}. In this paper, we focus on the class of functions $\Phi(x,t):=t^{p(x)}{(\log(e+t))}^{q(x)}$. For this class of functions $\Phi(x,t),$ the following embedding was established in \cite{HMOS10} for the space $W^{1,\Phi(\cdot,\cdot)}_0 (\Omega)$.

\begin{theorem}[\cite{HMOS10}]\label{embedding theorem in Rn}
Assume $p$ satisfies conditions $(p1)$ and $(p2)$, $q$ satisfies conditions $(q1)$ and $(q2)$, and $p^-+q^-\geq 1$. If $p^+<n,$ then for every $u\in W^{1,\Phi(\cdot,\cdot)}_0 (\Omega),$
\begin{center}
       $||u||_{L^{\Psi(\cdot,\cdot)}(\Omega)}\leq C|| u||_{W^{1,\Phi(\cdot,\cdot)}(\Omega)},$
\end{center}
where $\Phi(x,t):=t^{p(x)}{(\log(e+t))}^{q(x)}$ and $\Psi(x,t):=t^{p^*(x)}{(\log(e+t))}^{q(x)p^*(x)/p(x)}.$ Here $p^*(x)$ denotes the Sobolev conjugate of $p(x),$ that is  $1/p^*(x) =1/p(x)- 1/n$.
\end{theorem}
\noindent Here we establish the embedding for $W^{1,\Phi(\cdot,\cdot)} (\Omega)$ for bounded domains $\Omega$ with Lipschitz boundary. 
 \begin{theorem}\label{main embedding}
   Let $\Omega$ be an open and bounded set with Lipschitz boundary. Suppose that the exponent $p:\Omega \rightarrow [1,\infty)$ satisfies log-H\"older continuous with $1\leq p^-\leq p^+<n,$ the exponent $q:\Omega \rightarrow(-\infty,\infty)$ is log-log-H\"older continuous and $p^-+q^-\geq 1$. Consider $\Phi(x,t):=t^{p(x)}(\log(e+t))^{q(x)}$ and $\Psi(x,t):=t^{p^*(x)}(\log(e+t))^{q(x)p^*(x)/p(x)},$ 
where $1/p^*(x) =1/p(x)-1/n$. Then there exists a constant $C$ such that whenever $u\in W^{1,\Phi(\cdot,\cdot)}(\Omega),$
\begin{equation}\label{embedding}
    ||u||_{L^{\Psi(\cdot,\cdot)}(\Omega)}\leq C|| u||_{W^{1,\Phi(\cdot,\cdot)}(\Omega)}.
\end{equation}
\end{theorem}
The condition $p^-+q^-\geq 1$ in the above theorems guarantees that the function $\Phi(x,t)$ is a generalized weak $\Phi$-function (see Definition \ref{weak phi function} and Lemma \ref{sufficent condition for ainc}). When $p(x)\equiv p$ and $q(x)\equiv q, $ our result partially recovers the Orlicz-Sobolev embedding for $\Phi(t)=t^p\log^q(e+t),$ see \cite[Theorem 2 and Example 2]{Cia96}. Note that the condition $p^-+q^-\geq 1$ does not allow us to recover the case $p>1, q<0$ fully. \\

For the necessary part, it was shown in \cite{GKP21, GKP23} that $\Omega$ must satisfy the measure density condition to have the embedding $W^{1,p(\cdot)}(\Omega)\hookrightarrow L^{p^*}(\Omega),$ if $p$ satisfies the log-H\"older condition. Note that $\Omega$ satisfies measure density condition if there exists a constant $c>0$ such that for every $x$ in $\bar{\Omega}$ and each $R$ in $]0,1/2],$ one has $| B_R(x)\cap \Omega|\geq cR^n.$ This condition first appeared as a necessary condition for Sobolev embedding in \cite{HKT08} and later appeared as the same for other Sobolev-type embeddings as well, \cite{AYY22, AYY24, Kar19, Kar20}. Recently, a weaker version of the measure density condition, namely log-measure density condition, has appeared in \cite{HK22} as a necessary condition for certain Orlicz-Sobolev embeddings and also in \cite{GKP23} as a necessary condition for Sobolev-type embedding of $W^{1,p(\cdot)}(\Omega)$ if $p$ is log-log-H\"older continuous on $\Omega.$

\begin{definition} \label{log-s} A subset $\Omega$ of $\mathbb{R}^n$ is said to satisfy the log $s$-measure density condition if there exist two positive constants \( c \) and \( \alpha \) such that 
 for every $x$ in $\bar{\Omega}$ and each $R$ in $]0,1/2]$ one has
\[
c R^s (\log (\frac{1}{R}))^{-\alpha} \leq | B_R(x)\cap \Omega|.
\]
If $s=n$, one says that \( \Omega \) satisfies the log-measure density condition.
\end{definition}
\noindent Here we prove that if the embedding holds, then $\Omega$ satisfies the log-measure density condition. 
\begin{theorem} \label{main theorem}
Let $\Omega\subset\mathbb R^n$ be open. Consider $\Phi(x,t):=t^{p(x)}{(\log(e+t))}^{q(x)},$ and $\Psi(x,t):=t^{p^*(x)}{(\log(e+t))}^{q(x)p^*(x)/p(x)} $
where $p,q$ satisfies the condition $p^-+q^-\geq 1$, and $p^*(x)$ denotes the Sobolev conjugate of $p(x),$ that is, $1/p^*(x) =1/p(x)-1/n$. Suppose that
\begin{enumerate}
    \item[1.] The exponent $p(\cdot )$ is log-H\"older continuous with $p^+<n$,
    \item[2.] $W^{1,\Phi(\cdot,\cdot)}(\Omega) \hookrightarrow L^{\Psi(\cdot,\cdot)}(\Omega) .$
\end{enumerate}
Then $\Omega$ satisfies the log-measure density condition.
\end{theorem}
In this theorem, the condition $p^-+q^-\geq 1$ appears as a technical tool to find norm estimates of a characteristic function of a set. Also note that, here we do not require the log-log-H\"older continuity of $q$, unlike Theorem \ref{main embedding}. 

\section{Notations and Preliminary Results}

\begin{definition}
 A function $f:(0,\infty) \rightarrow \mathbb{R}$ is almost increasing if there exists a constant $a\geq 1$ such that $f(s) \leq af(t)$ for all $0< s< t .$ Similarly, a function $f:(0,\infty) \rightarrow \mathbb{R}$ is almost decreasing if there exists a constant $b\geq 1$ such that $bf(s) \geq f(t)$ for all $0< s< t .$ 
 \end{definition}
 \begin{definition}
Let $f:(0,\infty)\rightarrow\mathbb{R}$ and $p,q>0$.We say that f satisfies
\begin{enumerate}
\item[(i)]  $(Inc)_{p}$      if $f(t)/t^p$ is increasing; 
\item[(ii)] $(aInc)_{p}$     if $f(t)/t^p$ is almost increasing;
\item[(iii)]$(Dec)_{q}$      if $f(t)/t^q$ is decreasing;
\item[(iv)] $(aDec)_{q}$     if $f(t)/t^q$ is almost decreasing.
\end{enumerate}
\end{definition}
\begin{definition}
     We say that the exponent $p(\cdot)$ satisfies the log-H\"older decay condition if there exist $p_{\infty}\in \mathbb{R}$ and a constant $c_2>0$ such that 
    \begin{center}
    $|p(x)-p_{\infty}|\leq \frac{c_2}{\log(e+|x|)}$ for all $x\in\mathbb R^n$.
    \end{center}
\end{definition}
\begin{definition}
    We say that the exponent $p(\cdot)$ satisfies Nekvinda's decay condition, if there exists $c_1\in(0,1)$ and $p_{\infty}\in[1,\infty]$ such that 
   \begin{center}
       $\int_{(p(x)\not=p_{\infty})} c_1^{\frac{1}{|\frac{1}{p(x)}-\frac{1}{p_{\infty}}|}}dx < \infty$.
   \end{center}
\end{definition}
\begin{definition}\label{weak phi function}
   Let $(\Omega,\Sigma,\mu)$ be a $\sigma$-finite, complete measure space. A function $\Phi: \Omega\times [0,\infty)\rightarrow [0,\infty]$ is said to be a (generalized) $\Phi$-prefunction on $(\Omega,\Sigma,\mu)$ if $x\rightarrow \Phi(x,|f(x)|)$ is measurable for every $f\in L^0(\Omega,\mu) $ and $\Phi(x,.)$ is a $\Phi$-prefunction for $\mu$-almost every $x\in \Omega $. We say that the $\Phi$-prefunction $\Phi$ is  a (generalized weak) $\Phi$-function if it satisfies $(aInc)_1$. 
The sets of generalized weak $\Phi$-function is denoted by    $\Phi_w(\Omega,\mu)$.
\end{definition} 
\begin{lemma}\label{sufficent condition for ainc}
If $p^-+q^-\geq 1$, then the function $\Phi(x, t)=t^{p(x)}(\log(e+t))^{q(x)}$ satisfies $(aInc)_1$ and therefore we have $\Phi\in \Phi_w(\Omega).$
\end{lemma}

\begin{proof}
For $x\in\Omega$, define
\[
F_x(t)=\frac{\Phi(x,t)}{t}
=t^{p(x)-1}(\log(e+t))^{q(x)},\qquad t>0.
\]
A simple computation gives
\[
\frac{t F_x'(t)}{F_x(t)}
=p(x)-1+
q(x)\frac{t}{(e+t)\log(e+t)}.
\]
Note that $0<\frac{t}{(e+t)\log(e+t)}<1$ and $0\leq p^--1+q^-\leq p(x)-1+q(x)$ hence we have 
\[
\frac{t F_x'(t)}{F_x(t)}
\geq \big(p(x)-1+
q(x)\big)\frac{t}{(e+t)\log(e+t)}\geq 0,
\]
hence, $F_x(t)$ is increasing. Therefore, $\Phi$ satisfies $(aInc)_1$.
\end{proof}
\begin{definition}
     We say that $\Phi\in\Phi_w(\Omega,\mu)$ satisfies $(A0)$ if there is a constant $\beta\in(0,1]$ such that
   \begin{center} 
     $\beta\leq\Phi^{-1}(x,1)\leq{1/\beta}$  \quad for $\mu $-almost every $x\in \Omega$.
     \end{center}
\end{definition}
\begin{definition}
     Let $\Phi\in\Phi_w(\Omega,\mu)$. We say that $\Phi$ satisfies (A1) if there exists $\beta\in(0,1)$ such that 
     \begin{center}
      $\beta\Phi^{-1}(x,t)\leq\Phi^{-1}(y,t)$   
     \end{center}
     for every $t\in[1,\frac{1}{|B|}]$, almost every $x,y\in B\cap\Omega$ and every ball $B $ with $|B|\leq 1$.
\end{definition}
\begin{definition}
     We say that $\Phi\in\Phi_w(\Omega,\mu)$ satisfies $(A2)$ if for every $s>0$ there exists $\beta\in(0,1]$ and $h\in L^1(\Omega)\cap L^{\infty}(\Omega)$ such that \begin{center}
      $\beta\Phi^{-1}(x,t)\leq\Phi^{-1}(y,t)$   
      \end{center}
      for almost every $x,y\in \Omega$ and every $t\in[h(x)+h(y),s]$.
\end{definition}
Here, we prove an elementary result for a class of functions satisfying the above conditions.
 \begin{lemma}\label{phi satisfies all the condition}
 Let $\Omega \subset \mathbb R^n$ be a domain. If $p(\cdot)$ satisfies $(p1),(p2)$ and Nekvinda's decay condition and $q(\cdot)$ satisfies $(q1),(q2)$ 
 then the function $\Phi(x,t):= t^{p(x)}{(\log(e+t))}^{q(x)}$ satisfies $(A0), (A1), (A2)$ and $(Dec)_Q$, where $Q= p^++\max\{q^+, 0\} \geq 1.$
\end{lemma}
\begin{proof} First we will prove the condition $(A0),$ by splitting three different cases.\\
\textbf{Case-1} ($q(x) \geq 0$ for all $x\in \Omega$):
It is easy to verify the existence of constants $c_1$ and $c_2$ such that
$$\frac{c_2t^{\frac{1}{p(x)}}}{(\log(e+t))^{\frac{q(x)}{p(x)}}}\leq\Phi^{-1}(x,t)\leq \frac{c_1t^{\frac{1}{p(x)}}}{(\log(e+t))^{\frac{q(x)}{p(x)}}}$$
and hence 
$$\frac{c_2}{2^{\frac{q^+}{p^-}}}\leq \frac{c_2}{(\log(e+1))^{\frac{q(x)}{p(x)}}} \leq \Phi^{-1}(x,1)\leq\frac{c_1}{(\log(e+1))^{\frac{q(x)}{p(x)}}} \leq c_1 .$$ 
Now we can choose  $c_1\geq 1$ so that $c_1c_2 \geq 2^{\frac{q^+}{p^-}}$. Then (A0) follows by choosing $\beta =\frac{1}{c_1}.$\\
\textbf{Case-2} ($q(x)=-r(x)$ with $r(x) \geq 0$ for all $x\in \Omega$): It is easy to verify the existence of constants $c_3$ and $c_4$ such that
$$\frac{c_4t^{\frac{1}{p(x)}}}{(\log(e+t))^{\frac{q(x)}{p(x)}}}\leq\Phi^{-1}(x,t)\leq \frac{c_3t^{\frac{1}{p(x)}}}{(\log(e+t))^{\frac{q(x)}{p(x)}}},$$
hence $$c_4\leq \frac{c_4}{(\log(e+1))^{\frac{-r(x)}{p(x)}}} \leq \Phi^{-1}(x,1)\leq\frac{c_3}{(\log(e+1))^{\frac{-r(x)}{p(x)}}} \leq c_32^{\frac{r^+}{p^-}} = \frac{c_3}{2^{\frac{q^-}{p^-}}}.$$ 
Now we can choose  $c_4\leq 1$ so that $c_3 c_4\leq  2^\frac{q^-}{p^-}$. Then (A0) follows by choosing $\beta ={c_4}.$\\
 \textbf{Case-3} ($q(x) < 0$ for some $x\in \Omega$ and $q(x)\geq 0$ for some $x\in \Omega$): It is easy to verify the existence of constants $c_5$ and $c_6$ such that
$$\frac{c_6t^{\frac{1}{p(x)}}}{(\log(e+t))^{\frac{q(x)}{p(x)}}}\leq\Phi^{-1}(x,t)\leq \frac{c_5t^{\frac{1}{p(x)}}}{(\log(e+t))^{\frac{q(x)}{p(x)}}},$$
and hence $$\frac{c_6}{2^{\frac{q^+}{p^-}}}\leq \frac{c_6}{(\log(e+1))^{\frac{q(x)}{p(x)}}} \leq \Phi^{-1}(x,1)\leq\frac{c_5}{(\log(e+1))^{\frac{q(x)}{p(x)}}} \leq \frac{c_5}{\log^{\frac{q^-}{p^+}}(e+1)}.$$ 
Now we can choose  $c_6\leq 1$ so that $c_5 c_6\leq  2^\frac{q^+}{p^-} \log^{\frac{q^-}{p^+}}(e+1)$. Then (A0) follows by choosing $\beta =\frac{c_6}{2^\frac{q^+}{p^-}},$ and this completes the proof of the condition $(A0).$ 
\\
To show the condition $(A1),$ by symmetry, we may assume that $p(x)<p(y)$. If $t\in [1,\frac{1}{|B|}]$, then
    \begin{eqnarray*}
    \frac{\Phi^{-1}(x,t)}{\Phi^{-1}(y,t)}&\leq& c_1  t^{\frac{1}{p(x)}-\frac{1}{p(y)}}(\log(e+t))^{\frac{q(y)}{p(y)}-\frac{q(x)}{p(x)}} 
    \leq c_1  t^{\frac{1}{p(x)}-\frac{1}{p(y)}}(\log(e+t))^{|\frac{q(y)}{p(y)}-\frac{q(x)}{p(x)}|}\\
    &\leq &c_1 t^{\frac{1}{p(x)}-\frac{1}{p(y)}}(\log(e+t))^{|q(y)-q(x)|}\\
    &\leq & c_1|B|^{\frac{-c}{\log(e+\frac{1}{|x-y|})}}(\log(e+\frac{1}{|B|}))^{\frac{C}{ \log(e+\log(e+\frac{1}{|x-y|}))}}\\
    &\leq& c_1 e^{\frac{cn \log{\frac{1}{|x-y|}}}{\log(e+\frac{1}{|x-y|})}} { e^{\frac{c_0\log(\log(e+\frac{1}{|x-y|^n}))}{\log(e+\log(e+\frac{1}{|x-y|}))}}}\\
    &\leq& c_1 e^{cn} e^{ c_0+\frac{c_0\ln n}{\ln(\ln(e+1))}}.
    \end{eqnarray*}
    This yields $\beta\Phi^{-1}(x,t) \leq \Phi^{-1}(y,t)$ where $\frac{1}{\beta} = c_1 e^{cn} e^{ c_0+\frac{c_0 \ln n}{\ln(\ln(e+1))}}$ and hence (A1) follows.\\
    
   To show the condition $(A2),$ first note that since the function $p(\cdot)$ satisfies Nekvinda's decay condition, there exist $c_1\in(0,1)$ and $p_{\infty}\in[1,\infty]$ such that 
   \begin{center}
       $\int_{(p(x)\not=p_{\infty})} c_1^{\frac{1}{|\frac{1}{p(x)}-\frac{1}{p_{\infty}}|}}dx < \infty$.
   \end{center}
   \textbf{Case-1} ($q(x) \geq 0$ for all $x\in \Omega$): For this case, take $s=1,$ $\phi_{\infty}(t)= t^{p_{\infty}}$ and $\beta \leq 1.$ Note that $ \phi_{\infty}(t)\leq 1$ implies $t\leq 1$. We will consider two cases. In the points where $p(x)< p_{\infty},$ by Young's inequality 
    \begin{eqnarray*}
        \Phi(x, \beta t)&\leq& \log^{q^+}(e+1)\beta^{p(x)} t^{p(x)} \\
        &\leq& \frac{p(x)}{p_{\infty}}t^{p_{\infty}} + \frac{p_{\infty} -p(x)}{p_{\infty}}(\log^{q^+}(e+1))^{\frac{p_{\infty}}{p_{\infty} -p(x)}}\beta^{\frac{1}{|\frac{1}{p(x)}-\frac{1}{p_{\infty}}|}}\\
        &\leq& \phi_{\infty}(t)+ (\beta\log^{q^+}(e+1))^{\frac{1}{|\frac{1}{p(x)}-\frac{1}{p_{\infty}}|}}.\\
    \end{eqnarray*}
    Let us take
    \begin{center}
    $\beta < c_1[\log^{q^+}(e+1)]^{-1},$\\
     $h(x) = (\beta[\log^{q^+}(e+1)])^{\frac{1}{|\frac{1}{p(x)}-\frac{1}{p_{\infty}}|}}$,
     \end{center}
     where $c_1\in(0, 1)$ is the constant of Nekvinda's decay condition of $p.$ Then we have $h\in L^1(\Omega)\cap L^{\infty}(\Omega)$. In the points where $p(x)\geq p_{\infty},$ by taking the same choice of $\beta$ we have 
     \begin{equation*}
         \Phi(x, \beta t) \leq \log^{q^+}(e+1)\beta t^{p_{\infty}} \leq \phi_{\infty}(t)\leq \phi_{\infty}(t) + h(x).
     \end{equation*}
     We prove the other inequality similarly: In the points where $p(x)\leq  p_{\infty},$ as $\beta\leq 1,$
    \begin{equation*}
        \phi_{\infty}(\beta t)\leq t^{p(x)}\leq \Phi(x, t)\leq \Phi(x, t)+ h(x),
    \end{equation*}
    and in the points where $p(x)>p_{\infty}$, using Young's inequality
    \begin{equation*}
         \phi_{\infty}(\beta t)\leq \frac{p_{\infty}}{p(x)}t^{p(x)} + \frac{{p(x)}-p_{\infty} }{{p(x)}}\beta^{\frac{1}{|\frac{1}{p(x)}-\frac{1}{p_{\infty}}|}}\leq \Phi(x, t) + h(x),
    \end{equation*}
    which proves that $\Phi(x, t)$ satisfies $(A2)'.$
    
    \textbf{Case-2} ($q(x)=-r(x)$ with $r(x) \geq 0$ for all $x\in \Omega$): For this case, take $s=1,$ $\phi_{\infty}(t)= t^{p_{\infty}}$ and $\beta \leq 1.$ Note that $ \phi_{\infty}(t)\leq 1$ implies $t\leq 1$. We will consider two cases. In the points where $p(x)< p_{\infty},$ by Young's inequality 
     \begin{eqnarray*}
        \Phi(x, \beta t)&\leq& \log^{q(x)}(e+\beta t)\beta^{p(x)} t^{p(x)} \\
        &\leq& \beta^{p(x)} t^{p(x)}\leq\frac{p(x)}{p_{\infty}}t^{p_{\infty}} + \frac{{p_{\infty}}-p(x) }{{p_{\infty}}}\beta^{\frac{1}{|\frac{1}{p(x)}-\frac{1}{p_{\infty}}|}}\leq \phi_{\infty}(t) + h(x).
    \end{eqnarray*}
    Let us take 
    \begin{center}
    $\beta <c_1\log^{q^-}(e+1)$ and
   $ h(x) = c_1^{\frac{1}{|\frac{1}{p(x)}-\frac{1}{p_{\infty}}|}},$
    \end{center}
    where $c_1\in(0, 1)$ is the constant of Nekvinda's decay condition of $p.$ Then we have $h\in L^1(\Omega)\cap L^{\infty}(\Omega)$. In the points where $p(x)\geq p_{\infty},$ by taking the same choice of $\beta$ we have 
     \begin{equation*}
         \Phi(x, \beta t) = \log^{q(x)}(e+\beta t)\beta^{p(x)} t^{p(x)} \leq t^{p(x)} \leq \phi_{\infty}(t) \leq \phi_{\infty}(t) + h(x).
     \end{equation*}
      We do the other inequality similarly. In the points where $p(x)\leq  p_{\infty},$ as $\beta\leq 1$
      \begin{equation*}
        \phi_{\infty}(\beta t)= \beta^{p_{\infty}}t^{p_{\infty}}\leq t^{p(x)}\beta \leq t^{p(x)}\log^{q^-}(e+1)\leq  \Phi(x, t)\leq \Phi(x, t)+ h(x),
    \end{equation*}
    and in the points where $p(x)>p_{\infty}$, using Young's inequality
    \begin{equation*}
         \phi_{\infty}(\beta t)\leq c_1^{p_{\infty}} (\log^{q(x)}(e+t))^{p_{\infty}}t^{p_{\infty}} \leq \frac{p_{\infty}}{p(x)}t^{p(x)} \log^{q(x)}(e+t)+ \frac{{p(x)}-p_{\infty} }{{p(x)}}c_1^{\frac{1}{|\frac{1}{p(x)}-\frac{1}{p_{\infty}}|}}\leq \Phi(x, t) + h(x),
    \end{equation*}
    which proves that $\Phi(x, t)$ satisfies $(A2)'.$\\
    
    \textbf{Case-3} ($q(x) < 0$ for some $x\in \Omega$ and $q(x)\geq 0$ for some $x\in \Omega$): We can do this case similarly to the previous two cases by taking $s=1,$ $\phi_{\infty}(t)= t^{p_{\infty}}$, $\beta <c_1\frac{\log^{q^-}(e+1)}{\log^{q^+}(e+1)}$ and $ h(x) = c_1^{\frac{1}{|\frac{1}{p(x)}-\frac{1}{p_{\infty}}|}}$.\\
      
    \noindent Hence, by Lemma 4.2.7 of \cite{HH19}, $(A2)$ follows for all the cases of $q(\cdot)$ (see also \cite{HHS}).\\

\noindent Finally, the condition $(Dec)_Q$ follows easily, since we have for $0\leq s\leq t,$ 
\begin{equation*}
    \frac{\Phi(x,t)}{t^Q} = t^{p(x)-Q}{(\log(e+t))}^{q(x)}\leq s^{p(x)-Q}{(\log(e+s))}^{q(x)}=\frac{\Phi(x,s)}{s^Q}.\qedhere
\end{equation*}
\end{proof}

\begin{remark} \label{remark on bounded domain}
The above theorem is true for any domain, but for the bounded domain, the conditions $(p1)$ and $(p2)$ on $p(\cdot)$ are enough to prove that $\Phi$ satisfies $(A2),$ by taking $h = (\log^{q^+}(e+1))^{p^+ +1}$. 
\end{remark}
 
\begin{definition}
Let $\Phi\in\Phi_w(\Omega,\mu)$ and let $\rho_\Phi $ be given by
\begin{center}
  $\rho_\Phi(f)  := \int_\Omega\Phi(x,|f(x)|)d\mu(x)$
\end{center}	
for all $f\in L^0(\Omega,\mu)$. The function $\rho_{\Phi}$ is called a modular. The set 
\begin{center}
$ L^{\Phi(\cdot,\cdot)}(\Omega)$ := $\{f\in L^0(\Omega,\mu) :\rho_\Phi(\lambda f)< \infty  $ for some $\lambda > 0  \}$  
 \end{center}
 is called a generalized Orlicz space or Musielak-Orlicz (M-O) space.
 \end{definition}
 \begin{definition}
 Let $\Phi\in\Phi_w(\Omega,\mu)$. The function $u\in L^{\Phi(\cdot,\cdot)}(\Omega)\cap L^1_{loc}(\Omega)$ belongs to the Musielak-Orlicz-Sobolev space $W^{1,\Phi(\cdot,\cdot)}(\Omega)$ if its weak partial derivatives $\partial_\alpha u$ exist and belong to $L^{\Phi(\cdot,\cdot)}(\Omega)$ for all $|\alpha|\leq 1$. We define a semi-modular on $W^{1,\Phi(\cdot,\cdot)}(\Omega)$ by 
 \begin{center}
     $\rho_{W^{1,\Phi(\cdot,\cdot)}(\Omega)}(u) := \sum_{0\leq|\alpha|\leq 1}\rho_{\Phi} (\partial_\alpha u)$.
 \end{center}
 It induces a (quasi-) norm
 \begin{center}
$ ||u||_{W^{1,\Phi(\cdot,\cdot)} (\Omega)} := \inf  \{ \lambda >0 : \rho_ {W^{1,\Phi(\cdot,\cdot)} (\Omega)}\left(\frac{u}{\lambda}\right)\leq 1\}$.
 \end{center}
 \end{definition}
 
 Here, we prove three lemmas to estimate the norm of the characteristic function of a measurable set, considering three different sets of values of $q$.
 \begin{lemma}\label{Norm 1 lemma}
Let $\Phi:\Omega\times[0,\infty)\rightarrow[0,\infty)$  be given by $\Phi(x,t):=t^{p(x)}(\log(e+t))^{q(x)}$ with $q(x)\geq 0$ for all $x$ and $A\subset \Omega$ is a measurable set. Then 
\begin{equation}\label{Norm 1 lemma proof}
     \min \{|A|^\frac{1}{p_{A}^+},|A|^\frac{1}{p_{A}^-} \} \leq \|1_{A}\|_{L^{\Phi(\cdot,\cdot)}(\Omega)} \leq \max \{ |A|^\frac{1}{p_{A}^+} (\log (e+\frac{1}{|A|}))^{q_{A}^+} ,|A|^\frac{1}{p_{A}^-} (\log(1+e))^{q_{A}^+} \}.
\end{equation}  
\end{lemma}
\begin{proof}
We start with the proof of the second inequality of \eqref{Norm 1 lemma proof}. Let $u>|A|$ and assume first that $u\leq 1$. Then 
\begin{eqnarray*}
\bigintsss_A \Phi\bigg(x,\frac{1}{u^{\frac{1}{p_{A}^+}}(\log(e+\frac{1}{u}))^{q_{A}^+}}\bigg)dx &= & \bigints_A \frac{\bigg(\log\Big(e+\frac{1}{u^{\frac{1}{p _{A}^+}}\big(\log(e+\frac{1}{u})\big)^{q_{A}^+}}\Big)\bigg)^{q(x)}}{u^{\frac{p(x)}{p_{A}^+}}(\log(e+\frac{1}{u}))^{p(x)q_{A}^+}} dx \\
&\leq& \frac{|A|\bigg(\log\Big(e+\frac{1}{u^{\frac{1}{p_{A}^+}}(\log(e+\frac{1}{u}))^{q_{A}^+}}\Big)\bigg)^{q_{A}^+}}{u(\log(e+\frac{1}{u}))^{q_{A}^+}}\\
&< & 1,
\end{eqnarray*}
where in the final inequality we have used the fact that $u^{{\frac{1}{p_{A}^+}}-1}(\log(e+\frac{1}{u}))^{q_{A}^+} \geq 1 $. Hence we have $\|1_{A}\|_{L^{\Phi(\cdot,\cdot)}(\Omega)} \leq u^\frac{1}{p_{A}^+} (\log (e+\frac{1}{u}))^{q_{A}^+}$.
If $u>1$, we can similarly show that $\|1_{A}\|_{L^{\Phi(\cdot,\cdot)}(\Omega)} \leq u^\frac{1}{p_{A}^-} (\log(1+e))^{q_{A}^+}$. The second inequality follows as $u \rightarrow |A|^+$.

Let us then prove the first inequality of \eqref{Norm 1 lemma proof}. Let $u<|A|$ and assume first that $u\leq 1$. Then 
\begin{equation}
    \bigintsss_A \Phi\Big(x,\frac{1}{u^{\frac{1}{p_{A}^-}}}\Big)dx = \bigintss_A \frac{\Big(\log\big(e+\frac{1}{u^{\frac{1}{p_{A}^-}}}\big)\Big)^{q(x)}}{u^{\frac{p(x)}{p_{A}^-}}} dx \geq \frac{|A|}{u} >1.
\end{equation}
 Hence, we get $u^\frac{1}{p_{A}^-}  \leq \|1_{A}\|_{L^{\Phi(\cdot,\cdot)}(\Omega)}$. If $u>1$, we can similarly show that $u^\frac{1}{p_{A}^+}  \leq \|1_{A}\|_{L^{\Phi(\cdot,\cdot)}(\Omega)}$. The first inequality follows as $u \rightarrow |A|^-$.
\end{proof}
\begin{lemma}\label{Norm 2 lemma}
Let $\Phi:\Omega\times[0,\infty)\rightarrow[0,\infty)$  be given by $\Phi(x,t):=t^{p(x)}(\log(e+t))^{q(x)}$ with $q(x) < 0$ for all $x$. Assume that $p_{A}^- + q_{A}^- \geq 1$ and $A\subset \Omega$ is a measurable set with $|A|< \frac{1}{2}$. Then 
\begin{equation}\label{Norm 2}
     |A|^\frac{1}{p_{A}^-} (\log (e+\frac{1}{|A|}))^\frac{q_{A}^-}{p_{A}^-}  \leq \|1_{A}\|_{L^{\Phi(\cdot,\cdot)}(\Omega)} \leq  \max \{|A|^\frac{1}{p_{A}^+},|A|^\frac{1}{p_{A}^-} \}.
\end{equation}
\end{lemma}
\begin{proof}
We start with the proof of the second inequality of \eqref{Norm 2}. Let $u>|A|$  and assume that $u\leq 1$. Then we have
\begin{equation}
    \bigintsss_A \Phi\big(x,\frac{1}{u^{\frac{1}{p_{A}^+}}}\big)dx = \bigints_A \frac{\big(\log(e+\frac{1}{u^{\frac{1}{p_{A}^+}}})\big)^{q(x)}}{u^{\frac{p(x)}{p_{A}^+}}} dx \leq \frac{|A|}{u} <1,
\end{equation}
hence $u^\frac{1}{p_{A}^+}  \geq \|1_{A}\|_{L^{\Phi(\cdot,\cdot)}(\Omega)}$.  If $u>1$, we can similarly show that $u^\frac{1}{p_{A}^-}  \geq \|1_{A}\|_{L^{\Phi(\cdot,\cdot)}(\Omega)}$. The second inequality follows as $u \rightarrow |A|^+$.

Let us then prove the first inequality of \eqref{Norm 2}. Let $u<|A|$. Then 
\begin{eqnarray*}
\bigintss_A \Phi\bigg(x,\frac{1}{u^{\frac{1}{p_{A}^-}}(\log(e+\frac{1}{u}))^{\frac{q_{A}^-}{p_{A}^-}}}\bigg)dx &= & \bigints_A \frac{\bigg(\log(e+\frac{1}{u^{\frac{1}{p _{A}^-}}\big(\log(e+\frac{1}{u})\big)^{\frac{q_{A}^-}{p_{A}^-}}})\bigg)^{q(x)}}{u^{\frac{p(x)}{p_{A}^-}}(\log(e+\frac{1}{u}))^{p(x)\frac{q_{A}^-}{p_{A}^-}}} dx \\
&\geq& \frac{|A|\bigg(\log(e+\frac{1}{u^{\frac{1}{p_{A}^-}}\big(\log(e+\frac{1}{u})\big)^{\frac{q_{A}^-}{p_{A}^-}}})\bigg)^{q_{A}^-}}{u\big(\log(e+\frac{1}{u})\big)^{q_{A}^-}}\\
&>& 1,
\end{eqnarray*}
where we have used the fact that $u^{\frac{1}{p_{A}^-} -1}(\log(e+\frac{1}{u}))^{\frac{q_{A}^-}{p_{A}^-}} > 1$ under assumption $p_{A}^- + q_{A}^- \geq 1$ which follows from the increasing property of the function $h(t) = \frac{t^r}{(\log(e+t))^m}$ for all $t\geq 1$ when $r\geq m>0.$ 
Therefore we have $u^{\frac{1}{p_{A}^-}}(\log(e+\frac{1}{u}))^{\frac{q_{A}^-}{p_{A}^-}}\leq ||1_{A}||_{L^{\Phi(\cdot,\cdot)}(\Omega)}.$ The first inequality follows as $u \rightarrow |A|^-$.
\end{proof}

\begin{lemma}\label{Norm 3 lemma}
Let $\Phi:\Omega\times[0,\infty)\rightarrow[0,\infty)$  be given by $\Phi(x,t):=t^{p(x)}(\log(e+t))^{q(x)}$ where $q(x) < 0$ for some $x $ and $q(x)\geq 0$ for some $x$. Assume that $p_{A}^- + q_{A}^- \geq 1$ and $A\subset \Omega$ is a measurable set with $|A|< \frac{1}{2}$. Then, there exist constants $b_1>0$, $b_2>0$ such that
\begin{equation}\label{Norm 3}
     b_1|A|^\frac{1}{p_{A}^-} \big(\log (e+\frac{1}{|A|})\big)^\frac{q_{A}^-}{p_{A}^-}  \leq \|1_{A}\|_{L^{\Phi(\cdot,\cdot)}(\Omega)} \leq b_2 \max  \{ |A|^\frac{1}{p_{A}^+} \big(\log (e+\frac{1}{|A|})\big)^{q_{A}^+} ,|A|^\frac{1}{p_{A}^-} \big(\log(1+e)\big)^{q_{A}^+} \}.
\end{equation}
\end{lemma}
\begin{proof}
We start with the proof of the second inequality of \eqref{Norm 3}. Let $u>|A|$ and assume first that $u\leq 1$. Then 
\begin{eqnarray*}
\bigintsss_A \Phi\Big(x,\frac{1}{2u^{\frac{1}{p_{A}^+}}(\log\big(e+\frac{1}{2u}\big)\big)^{q_{A}^+}}\Big)dx &= & \bigintss_A  \frac{\bigg(\log\Big(e+\frac{1}{2u^{\frac{1}{p _{A}^+}}\big(\log\big(e+\frac{1}{2u}\big)\big)^{q_{A}^+}}\Big)\bigg)^{q(x)}}
{2^{p(x)} u^{\frac{p(x)}{p_{A}^+}}\big(\log(e+\frac{1}{2u})\big)^{p(x)q_{A}^+}} dx \\
&= &  \bigintss_{A\cap (x: q(x)\geq 0)}  \frac{\bigg(\log\Big(e+\frac{1}{2u^{\frac{1}{p _{A}^+}}\big(\log\big(e+\frac{1}{2u}\big)\big)^{q_{A}^+}}\Big)\bigg)^{q(x)}}
{2^{p(x)}u^{\frac{p(x)}{p_{A}^+}}\bigg(\log\big(e+\frac{1}{2u}\big)\bigg)^{p(x)q_{A}^+}} dx \\
&+ & \bigintss_{A\cap (x: q(x)< 0)} \frac{\bigg(\log\Big(e+\frac{1}{2u^{\frac{1}{p _{A}^+}}\big(\log\big(e+\frac{1}{2u}\big)\big)^{q_{A}^+}}\Big)\bigg)^{q(x)}}
{2^{p(x)}u^{\frac{p(x)}{p_{A}^+}}\big(\log\big(e+\frac{1}{2u}\big)\big)^{p(x)q_{A}^+}} dx \\
&\leq& \frac{1}{2}\frac{|A|\bigg(\log\Big(e+\frac{1}{2u^{\frac{1}{p_{A}^+}}\big(\log\big(e+\frac{1}{2u}\big)\big)^{q_{A}^+}}\Big)\bigg)^{q_{A}^+}}{u\big(\log\big(e+\frac{1}{2u}\big)\big)^{q_{A}^+}} + \frac{|A|}{2u}\\
&< & 1,
\end{eqnarray*} 
where in the final inequality, we have used the fact that $u^{{\frac{1}{p_{A}^+}}-1}(\log(e+\frac{1}{u}))^{q_{A}^+} \geq 1 $. Hence, we have $\|1_{A}\|_{L^{\Phi(\cdot,\cdot)}(\Omega)} \leq 2u^\frac{1}{p_{A}^+} (\log (e+\frac{1}{2u}))^{q_{A}^+}$.
If $u>1$, we can similarly show that $\|1_{A}\|_{L^{\Phi(\cdot,\cdot)}(\Omega)} \leq u^\frac{1}{p_{A}^-} (\log(1+e))^{q_{A}^+}$. The second inequality follows as $u \rightarrow |A|^+$.
Let us then prove the first inequality of \eqref{Norm 3}. Let $u<|A|$. Then 
\begin{eqnarray*}
& &\bigintss_A \Phi\bigg(x,\frac{1}{2^{\frac{1}{{p_A}^+}}u^{\frac{1}{p_{A}^-}}\big(\log(e+\frac{1}{u2^{1/{p_A}^+}})\big)^{\frac{q_{A}^-}{p_{A}^-}}}\bigg)dx \\
& = & \bigints_A \frac{\bigg(\log(e+\frac{1}{2^{\frac{1}{{p_A}^+}}u^{\frac{1}{p _{A}^-}}\big(\log(e+\frac{1}{u2^{1/{p_A}^+}})\big)^{\frac{q_{A}^-}{p_{A}^-}}})\bigg)^{q(x)}}{2^{\frac{p(x)}{{p_A}^+}}u^{\frac{p(x)}{p_{A}^-}}(\log(e+\frac{1}{u2^{1/{p_A}^+}}))^{p(x)\frac{q_{A}^-}{p_{A}^-}}} dx \\
&=& \bigintss_{A \cap (x: q(x) \geq 0)} \frac{\bigg(\log(e+\frac{1}{2^{\frac{1}{{p_A}^+}}u^{\frac{1}{p _{A}^-}}\big(\log(e+\frac{1}{u2^{1/{p_A}^+}})\big)^{\frac{q_{A}^-}{p_{A}^-}}})\bigg)^{q(x)}}{2^{\frac{p(x)}{{p_A}^+}}u^{\frac{p(x)}{p_{A}^-}}\big(\log(e+\frac{1}{u2^{1/{p_A}^+}})\big)^{p(x)\frac{q_{A}^-}{p_{A}^-}}} dx\\
&+ &\bigintss_{A \cap (x: q(x) < 0)}\frac{\bigg(\log(e+\frac{1}{2^{\frac{1}{{p_A}^+}}u^{\frac{1}{p _{A}^-}}\big(\log(e+\frac{1}{u2^{1/{p_A}^+}})\big)^{\frac{q_{A}^-}{p_{A}^-}}})\bigg)^{q(x)}}{2^{\frac{p(x)}{{p_A}^+}}u^{\frac{p(x)}{p_{A}^-}}\big(\log(e+\frac{1}{u2^{1/{p_A}^+}})\big)^{p(x)\frac{q_{A}^-}{p_{A}^-}}} dx \\
&\geq& \frac{|A|}{2u} + \frac{|A|\bigg(\log(e+\frac{1}{2^{\frac{1}{{p_A}^+}}u^{\frac{1}{p_{A}^-}}\big(\log(e+\frac{1}{u2^{1/{p_A}^+}})\big)^{\frac{q_{A}^-}{p_{A}^-}}})\bigg)^{q_{A}^-}}{2u\big(\log(e+\frac{1}{u2^{1/{p_A}^+}})\big)^{q_{A}^-}}\\
&>& 1,
\end{eqnarray*}
where in the final inequality we have used the fact that $u^{\frac{1}{p_{A}^-} -1}(\log(e+\frac{1}{u2^{1/{p_A}^+}}))^{\frac{q_{A}^-}{p_{A}^-}} > 1$ when  $p_{A}^- + q_{A}^- \geq 1.$ Therefore we have $2^{\frac{1}{{p_A}^+}}u^\frac{1}{p_{A}^-} (\log (e+\frac{1}{u2^{1/{p_A}^+}}))^\frac{q_{A}^-}{p_{A}^-}  \leq \|1_{A}\|_{L^{\Phi(\cdot,\cdot)}(\Omega)}.$ The first inequality follows as $u \rightarrow |A|^-$.
\end{proof}

In the following lemma, we extend the exponent $p$ from $\Omega$ to $\mathbb{R}^n$; we will use the lemma only to extend the exponent functions here, preserving the modulus of continuity and upper and lower bounds, using the technique of Edmunds and R\'akosn\'ik \cite[Theorem 4.1]{ER} which was originally introduced by Hestenes \cite{Hestenes}. The same method was also used in \cite{Die04}. We recall the proof here for the convenience of the readers. Also, we will only use the lemma to extend the exponent functions.
  \begin{lemma}\label{extension of the exponents}
   Let $\Omega \subset \mathbb R^n$  be an open, bounded set with Lipschitz boundary. Let $p:\Omega \rightarrow (-\infty,\infty)$ satisfy the uniform continuity condition
   \begin{center}
       $|p(x)-p(y)|\leq \rho(|x-y|)$ \quad for all $x,y \in \Omega$
   \end{center}
   where $\rho$ is concave for $t\geq 0$ and $\rho(t) \rightarrow0$ for $t\rightarrow 0^+$. Then there exists an extension $p_1$ on $\mathbb R^n$ of $p$ and a constant $C>0$, such that 
   \begin{center}
    $|p_1(x)-p_1(y)|\leq \rho(C|x-y|)$ \quad for all $x,y \in \mathbb R^n$   
   \end{center}
   Moreover, there holds $p_1^-=p^-$ and $p_1^+=p^+$.
  \end{lemma} 
  \begin{proof}
      Let ${V_j}, j= 1, . . ., k,$ be the covering of the boundary $\partial\Omega$ which corresponds to the local description of $\partial\Omega$. More precisely, for each $ j= 1, . . ., k,$ there is a local coordinate system $(x', x_n)$ such that
      \begin{center}
          $V_j= \{(x', x_n): |x_i|<\delta, i= 1, . . ., n-1,  {a_j}(x')-\beta<x_n<{a_j}(x')+\beta\},$
          $V_j\cap \Omega = \{x\in V_j : {a_j}(x')<x_n<{a_j}(x')+\beta\} $
      \end{center}
      and
      \begin{center}
          $\{x\in \bar{V_j} : x_n<{a_j}(x')\} \cap \bar{\Omega} = \varnothing, $
 \end{center}
    where $\beta, \delta$ are some fixed positive numbers and ${a_j} \in C^{0,1}((-\delta,\delta)^{n-1})$ are the functions describing the boundary. Define the mappings $T_j : (-\delta,\delta)^{n-1} \times (-\beta,\beta)\rightarrow \mathbb R^n, j= 1,\ldots ,n, $ 
    by
    \begin{center}
        $T_j(x',x_n) = (x', x_n+ a_j(x')).$
    \end{center}
    Then $T_j$ are bi-Lipschitz mappings.
    For these flattened domains ${{T_j}^{-1}}(V_j)$, define the reflection operator 
      \begin{equation*}
          Ef(x) = \begin{cases}
          f(x' ,x_n) \quad for \quad x_n\geq0,\\
           f(x' ,-x_n) \quad for \quad x_n<0
          \end{cases}
      \end{equation*}
  and the functions ${p_1}_j$ on $V_j\cup \Omega$ by
      \begin{equation*}
          {p_1}_j(x) = \begin{cases}
                   p(x)     \quad           for \quad x\in \Omega,\\
              Er_j({{T_J}^{-1}}(x))) \quad for \quad x\in V_j/\Omega,
          \end{cases}
      \end{equation*}
      where $r_j := p\circ T_j.$ Note that since $E, T_j,$ and ${T_j}^{-1}$ are Lipschitz, there exists $C>0$ such that 
      \begin{center}
         $|{{p_1}_j}(x)-{{p_1}_j}(y)| \leq \rho(C|x-y|)$ \quad for all $x,y \in \Omega$.
      \end{center}
      Then extend the functions ${p_1}_j$ on $\Omega$ to $\tilde{{p}_1}_j$ on $\mathbb R^n$ preserving their upper and lower bounds. Note that this extension is possible due to McShane \cite[Theorem 2 and Corollary 2]{McSHANE34} and the fact that $\rho$ is concave with $\rho(t) \rightarrow0$ for $t\rightarrow 0^+$. Define $p_1:\mathbb  R^n \rightarrow (-\infty,\infty)$ by 
      \begin{center}
         $ {p_1}(x) := \min\limits_{{j=1,...,k}} {\tilde{{p}_1}_j}(x) $ \quad for \quad $x\in \mathbb R^n.$ 
      \end{center}
      Thus, we have 
      \begin{center}
    $|p_1(x)-p_1(y)|\leq \rho(C|x-y|)$ \quad for all $x,y \in \mathbb R^n$.   
   \end{center}
 This proves the lemma.
  \end{proof}
  \begin{lemma} \label{p-extension}
  Let $\Omega \subset \mathbb R^n$  be an open, bounded set with Lipschitz boundary. Suppose that $p$ satisfies conditions $(p1)$ and $(p2),$ and $q$ satisfies conditions $(q1)$ and $(q2).$
  Then there exists an extension $p_1$ on $\mathbb R^n$ of $p$ with $p_1^-=p^-,$ $p_1^+=p^+$ and an extension $q_1$ on $\mathbb R^n$ of $q$ with $q_1^-=q^-$ and $q_1^+=q^+$, which satisfies the same local uniform continuity conditions (with possibly different constants). 
  \end{lemma}
  \begin{proof}
     Since the mapping $\rho:t\rightarrow C/\log(e+1/t)$ is concave for $t\geq 0$ and $\rho(t) \rightarrow0$ for $t\rightarrow 0^+$ and $p$ satisfies uniformly the local continuity condition such that 
     \begin{center}
        $|p(x)-p(y)|\leq \rho(|x-y|)$\quad for all $x,y \in \Omega$,  
     \end{center}
     Due to Lemma \ref{extension of the exponents}, it follows that there exists an extension $p_1$ on $\mathbb R^n$ of $p$, which possesses all the desired properties. The proof for the extension of $q$ is similar.
  \end{proof}
  \begin{theorem}[\cite{Juu}]  \label{Justi theorem}
      Let $\Omega$ be an $(\epsilon,\delta)$-domain with $rad(\Omega)>0.$ Suppose that $\Phi\in\Phi_w(\Omega,\mu)$ satisfies $(A0), (A1), (A2)$ and $(aDec)_{q}$ with $q\geq1.$ Let $\Psi\in\Phi_w(\mathbb R^n,\mu)$ be the extension of $\Phi$ which also satisfies $(A0), (A1), (A2)$ and $(aDec)_{q}$ with $q\geq1.$ There exists an operator $\Lambda : W^{1,\Phi(\cdot,\cdot)}(\Omega) \rightarrow W^{1,\Psi(\cdot,\cdot)}(\mathbb R^n)$ and a constant $B$ such that 
      \begin{equation*}
    ||\Lambda u||_{W^{1,\Psi(\cdot,\cdot)}(\mathbb{R}^n)}\leq B ||u||_{W^{1,\Phi(\cdot,\cdot)}(\Omega)},
\end{equation*}
for every $u\in W^{1,\Phi(\cdot,\cdot)}(\Omega) .$
  \end{theorem}
\noindent

 \section{Main Results}
\textbf{Proof of Theorem \ref{main embedding}:}
By Lemma \ref{p-extension}, there exist extensions
$\tilde p,\tilde q:\mathbb{R}^n\to\mathbb{R}$
such that $\tilde p|_\Omega=p,
\,
\tilde q|_\Omega=q,$
and also $p^-=\tilde p^-,$ $p^+=\tilde p^+,$ $q^-=\tilde q^-$ and $q^+=\tilde q^+.$ Moreover, $\tilde{p}$ and $\tilde{q}$ satisfy the conditions $(p2)$ and $(q2)$, respectively.
Using these extensions, we would like to define $p_1$ and $q_1$ on $\mathbb{R}^n$ that carry these properties and so that $p_1$ also satisfies Nekvinda's decay condition. This can be done using the boundedness of $\Omega.$ Indeed, there exists $R_0>0$ such that
$\overline{\Omega}\subset B(0,R_0).$ Choose a Lipschitz function
$\eta:\mathbb{R}^n\to[0,1]$ such that 
\[
\eta(x)=
\begin{cases}
1, & x\in B(0,R_0),\\
0, & x\notin B(0,2R_0).
\end{cases}
\]
Let $p_\infty\in[p^-,p^+]$, and $q_\infty\in[q^-,q^+].$ Define
\[{
p_1(x)
=
\eta(x)\tilde p(x)
+
(1-\eta(x))p_\infty
}\]
and
\[{
q_1(x)
=
\eta(x)\tilde q(x)
+
(1-\eta(x))q_\infty.\\
}\]
Note that, $ p_1|_\Omega=p,
\,
q_1|_\Omega=q,$ $p_1^-=p^-,
\,
p_1^+=p^+,$
and $q_1^-=q^-,
\,
q_1^+=q^+.$
It remains to show that $p_1$ satisfies condition $(p2)$ and Nekvinda's decay condition, while $q_1$ satisfies condition $(q2).$ Towards this end, for $x,y\in\mathbb{R}^n$, we have
\begin{align*}
|p_1(x)-p_1(y)|
&=
\big|
\eta(x)(\tilde p(x)-p_\infty)
-
\eta(y)(\tilde p(y)-p_\infty)
\big|\\
&\leq
\eta(x)|\tilde p(x)-\tilde p(y)| +
|\eta(x)-\eta(y)|
|\tilde p(y)-p_\infty|.
\end{align*}

Since $\eta$ is Lipschitz, there exists $L>0$
such that $|\eta(x)-\eta(y)|
\leq L|x-y|.$ Moreover, $|\tilde p(y)-p_\infty|
\leq p^+-p^-.$ Therefore,
\[
|p_1(x)-p_1(y)|
\leq
\frac{C_p}{\log(e+1/|x-y|)}
+
L(p^+-p^-)|x-y|.
\]
Note that there exists a positive constant $C_1$
such that
\[
|x-y|
\leq
\frac{C_1}{\log(e+1/|x-y|)}.
\]
Hence,
\[
|p_1(x)-p_1(y)|
\leq
\frac{C_2}{\log(e+1/|x-y|)}.
\]
Thus $p_1$ satisfies the log-H\"older
continuity condition on $\mathbb{R}^n$. Similarly, we can show that $q_1$ satisfies the log-log-H\"older continuity condition on $\mathbb{R}^n$. To show that $p_1$ satisfies Nekvinda's decay condition on $\mathbb{R}^n$, note that, by construction,
\[
p_1(x)=p_\infty
\qquad
\text{for every }|x|\geq2R_0.
\]
Consequently,
\[
\{x\in\mathbb{R}^n:p_1(x)\neq p_\infty\}
\subset B(0,2R_0).
\]
Let $c\in(0,1)$. Then
\[
\begin{aligned}
&\int_{\{p_1(x)\neq p_\infty\}}
c^{\frac{1}
{\left|\frac1{p_1(x)}-\frac1{p_\infty}\right|}}
\,dx
\leq
\int_{B(0,2R_0)}1\,dx=
|B(0,2R_0)|<\infty.
\end{aligned}
\]
Therefore, $p_1$ satisfies Nekvinda's decay condition. \\

Now, consider $\Phi_1(x,t):=t^{p_1(x)}\ {(\log(e+t))}^{q_1(x)}$ and ${\Psi_1}(x,t):=t^{{p_1}^*(x)}(\log(e+t))^{{q_1}(x){p_1}^*(x)/{p_1}(x)}.$ Note that $p_1$ satisfies the conditions $(p1),$ $(p2)$ and Nekvinda's decay condition, whereas $q_1$ satisfies the conditions $(q1)$ and $(q2),$ and therefore by Lemma \ref{phi satisfies all the condition}, $\Phi_1$ satisfies $(A0), (A1), (A2)$ and $(Dec)_{p_1^++\max\{q_1^+, 0\}}$. Since $p_1^-+q_1^-=p^-+q^-\geq 1$, we get by Theorem \ref{Justi theorem}, a linear extension operator $\mathcal{E}:W^{1,\Phi(\cdot,\cdot)}(\Omega) \rightarrow W^{1,\Phi_1(\cdot,\cdot)}(\mathbb R^n) $ and a constant $c_1$ such that 
\begin{equation}\label{from extension}
    ||v||_{W^{1,\Phi_1(\cdot,\cdot)}(\mathbb{R}^n)}\leq c_1  ||u||_{W^{1,\Phi(\cdot,\cdot)}(\Omega)},
\end{equation}
and $v|_{\Omega}=u$ for all $v\in W^{1,\Phi_1(\cdot,\cdot)}(\mathbb{R}^n),$ where $\mathcal{E}u=:v.$
On the other hand, using Theorem \ref{embedding theorem in Rn}, we get a constant $c_2$ such that
\begin{equation}\label{from embedding}
       ||v||_{L^{\Psi_1(\cdot,\cdot)}(\mathbb{R}^n)}\leq c_2||v||_{W^{1,\Phi_1(\cdot,\cdot)}(\mathbb{R}^n)},
\end{equation}
for all $v\in W^{1,\Phi_1(\cdot,\cdot)}_0(\mathbb{R}^n).$
Also, since $\Phi_1$ satisfies $(A0), (A1), (A2)$ and $(Dec)_{p_1^++\max\{q_1^+, 0\}}$, by Theorem 6.4.4 of \cite{HH19} we have $W^{1,\Phi_1(\cdot,\cdot)}_0(\mathbb R^n) =W^{1,\Phi_1(\cdot,\cdot)}(\mathbb R^n)$. Hence, the inequality \eqref{from embedding} holds for all $v\in W^{1,\Phi_1(\cdot,\cdot)}(\mathbb{R}^n).$ Finally, we use the inequalities \eqref{from extension}, \eqref{from embedding} and the facts that $\Phi_1,$ $\Psi_1$ are extensions of $\Phi$ and $\Psi$ respectively, which we obtain, for all $u\in W^{1,\Phi(\cdot,\cdot)}(\Omega),$ 
\begin{eqnarray*}
    ||u||_{L^{\Psi(\cdot,\cdot)}(\Omega)}=||v||_{L^{\Psi(\cdot,\cdot)}(\Omega)}\leq ||v||_{L^{\Psi_1(\cdot,\cdot)}(\mathbb{R}^n)}\leq c_2||v||_{W^{1,\Phi_1(\cdot,\cdot)}(\mathbb{R}^n)} &\leq & C ||u||_{W^{1,\Phi(\cdot,\cdot)}(\Omega)},
\end{eqnarray*}
where $C=c_1c_2.$ \qed
\bigskip

    Note that in the above proof, we are using Lemma \ref{p-extension} only to extend the functions $p$ and $q$, which satisfy the same local uniform continuity condition, and not using the boundedness of the linear extension operator. Using the result of a recent paper \cite{CD}, we can also prove Theorem \ref{main embedding}, even for bounded John domains. We give the details here as an alternative proof:\\
    
\noindent
\textbf{Alternate proof of Theorem \ref{main embedding}:}
The definitions of $H_n(x, t)$ and $\Phi_n(x,t)$ as introduced in \cite{CD}, are given by  
\begin{align*}
H_n(x,t)
&=
\left(
\int_0^t
\left(
\frac{s}{\Phi(x,s)}
\right)^{\frac1{n-1}}\,ds
\right)^{\frac{n-1}{n}}
\quad\text{and}\quad
\Phi_n(x,t)
=
\Phi\bigl(x,H_n^{-1}(x,t)\bigr).
\end{align*}
Note that, \cite[Theorem 3.6]{CD} states that if $\Omega$ is a bounded John domain, and $\Phi$ satisfies $(A0)$ and $(A1),$ then  $W^{1,\Phi(\cdot,\cdot)}(\Omega) \hookrightarrow L^{\Phi_n(\cdot,\cdot)}(\Omega) .$
In our case, by setting,
\[
a(x)=\frac{n-p(x)}{n-1}
\quad\text{and}\quad
b(x)=-\frac{q(x)}{n-1},
\]
we can write
\[
H_n(x,t)
=
\left(
\int_0^t
s^{a(x)-1}
(\log(e+s))^{b(x)}\,ds
\right)^{\frac{n-1}{n}}.
\]
Now, for $x\in\Omega,$ we define
\begin{equation}\label{Ix}
I_x(t)
=
\int_0^t
s^{a(x)-1}
(\log(e+s))^{b(x)}\,ds
\end{equation}
and we claim that
\begin{equation}\label{eq:integral-equivalence}
I_x(t)
\approx
t^{a(x)}(\log(e+t))^{b(x)} =: F_x(t)
\end{equation}
for  $t>0$. Indeed, by differentiating $F_x(t)$ with respect to $t$, we obtain
\begin{align*}
F_x'(t)
&=
a(x)t^{a(x)-1}
(\log(e+t))^{b(x)}+
b(x)t^{a(x)}
(\log(e+t))^{b(x)-1}
\frac1{e+t}\\
&=
t^{a(x)-1}
(\log(e+t))^{b(x)}
\left[
a(x)+
b(x)\frac{t}{(e+t)\log(e+t)}
\right].
\end{align*}
Since $p^+<n$, we have
\[
0<
\frac{n-p^+}{n-1}
\le a(x)
\le
\frac{n-p^-}{n-1}\quad\text{and}\quad b(x)
\le
\frac{-q^-}{n-1}.
\]
Since
\[0<
\frac{t}{(e+t)\log(e+t)}<1,
\]
hence $F_x'(t)\approx t^{a(x)-1}
(\log(e+t))^{b(x)} $ and therefore, by \eqref{Ix}, we have $I_x(t)\approx \int_0^t F_x'(s) \, ds = F_x(t) $
which proves
\eqref{eq:integral-equivalence}.

Using \eqref{eq:integral-equivalence},
we obtain
\begin{align*}
H_n(x,t)
&=
I_x(t)^{\frac{n-1}{n}}
\approx
\left(
t^{a(x)}(\log(e+t))^{b(x)}
\right)^{\frac{n-1}{n}}
=
t^{\frac{a(x)(n-1)}n}
(\log(e+t))^{\frac{b(x)(n-1)}n}.
\end{align*}
Substituting the values of $a(x)$
and $b(x)$ yields
\begin{equation*}
{
H_n(x,t)
\approx
t^{\frac{n-p(x)}n}
(\log(e+t))^{-\frac{q(x)}n}
}
\end{equation*}
and hence
\begin{equation*}
{
H_n^{-1}(x,t)
\approx
t^{\frac{n}{n-p(x)}}
(\log(e+t))^{\frac{q(x)}{n-p(x)}}.
}
\end{equation*}
Therefore,
\begin{align}\label{phin}
\Phi_n(x,t)
&=
\left(H_n^{-1}(x,t)\right)^{p(x)}
\left(
\log(e+H_n^{-1}(x,t))
\right)^{q(x)},
\end{align}
Using the estimate
\[
\left(
\log(e+H_n^{-1}(x,t))
\right)^{q(x)}
\approx
(\log(e+t))^{q(x)}
\]
in the equation \eqref{phin}, we get
\begin{align*}
\Phi_n(x,t)
\approx
t^{\frac{np(x)}{n-p(x)}}
(\log(e+t))^{
\frac{p(x)q(x)}{n-p(x)}+q(x)
}
& = 
t^{\frac{np(x)}{n-p(x)}}
(\log(e+t))^{\frac{nq(x)}{n-p(x)}}\\
& =  t^{p^*(x)}
(\log(e+t))^{\frac{q(x)p^*(x)}{p(x)}}.
\end{align*}
Consequently, from \cite[Theorem 3.6]{CD} it follows that $W^{1,\Phi(\cdot,\cdot)}(\Omega) \hookrightarrow L^{\Phi_n(\cdot,\cdot)}(\Omega)=L^{\Psi(\cdot,\cdot)}(\Omega).$\qed \\

Now, to prove Theorem \ref{main theorem}, we need the following lemma. 
 
 \begin{lemma} \label{main lemma }
If we have the same assumptions as in Theorem \ref{main theorem}, then there exist positive constants $C_1$, $C_2$, $C_3$ such that for all $x$ in $\bar{\Omega}$ and every $R$ in $]0,1]$ we have
\[
    R- \tilde{R} \leq C_1 |A_R|^{\frac{1}{n}+\frac{1}{p_{A_R}^+}-\frac{1}{p_{A_R}^-}} (\log (e+\frac{1}{|A_R|}))^{q_{A_R}^+} \label{Case-1}
\] 
when $q(x) \geq 0$ for all $x$,
\[
    R- \tilde{R} \leq C_2 |A_R|^{\frac{1}{n}+\frac{1}{p_{A_R}^+}-\frac{1}{p_{A_R}^-}} (\log (e+\frac{1}{|A_R|}))^{Q_{A_R}} \label{case-2}
\] 
when $q(x) < 0$ for all $x$, and 
\[
    R- \tilde{R} \leq C_3 |A_R|^{\frac{1}{n}+\frac{1}{p_{A_R}^+}-\frac{1}{p_{A_R}^-}} (\log (e+\frac{1}{|A_R|}))^{X_{A_R}} \label{case-3}
\] 
when  $q(x) < 0$ for some $x$ and $q(x)\geq 0$ for some $x$. 
\end{lemma}
\begin{proof}
   Since  $W^{1,\Phi(\cdot,\cdot)}(\Omega) \hookrightarrow L^{\Psi(\cdot,\cdot)}(\Omega),$ there exists a constant $c_1 >0$ such that whenever $u\in W^{1,\Phi(\cdot,\cdot)} (\Omega)$ one has the inequality
 \begin{equation} \label{inequality 1}
       ||u||_{L^{\Psi(\cdot,\cdot)}(\Omega)}\leq c_1|| u||_{W^{1,\Phi(\cdot,\cdot)}(\Omega)}.
 \end{equation}
For a fixed $x\in\bar{\Omega}$ let $u(y):=\phi(y-x),$ where $y\in\Omega$ and $\phi$ is a cut-off function so that
\begin{enumerate}
\item $\phi:\mathbb{R}^n \rightarrow [0,1]$, 
\item $\hbox{spt}\ \phi \subset B_R(0)$,
\item  $\phi|_{B_{\tilde{R}}(0)}=1$, and 
\item $|\nabla \phi | \leq \tilde{c} / (R-\tilde{R})$ for some constant $\tilde{c}$.
\end{enumerate}
\bigskip
Note that we have the inequalities
\[
\|1_{B_{\tilde{R}}}\|_{L^{\Psi(\cdot,\cdot)(\Omega)}}\leq \|u\|_{L^{\Psi(\cdot,\cdot)(\Omega)}}, \quad\quad \|u\|_{L^{\Phi(\cdot,\cdot)(\Omega)}}\leq \|1_{B_{R}}\|_{L^{\Phi(\cdot,\cdot)(\Omega)}}
\]
and 
\[
    \|\nabla u\|_{L^{\Phi(\cdot,\cdot)(\Omega)}} \leq \frac{\tilde{c}}{R-\tilde{R}}\|1_{B_R \thicksim B_{\tilde{R}}}\|_{L^{\Phi(\cdot,\cdot)(\Omega)}}
 \leq  \frac{\tilde{c}}{R-\tilde{R}}\ \| 1_{B_R}\|_{L^{\Phi(\cdot,\cdot)(\Omega)}} .
\]
Use these inequalities in inequality \eqref{inequality 1} to obtain
\begin{eqnarray*} 
\|1_{B_{\tilde{R}}}\|_{L^{\Psi(\cdot,\cdot)(\Omega)}} &\leq & c_1(\|1_{B_{R}}\|_{L^{\Phi(\cdot,\cdot)(\Omega)}}+\frac{\tilde{c}}{R-\tilde{R}}\ \| 1_{B_R}\|_{L^{\Phi(\cdot,\cdot)(\Omega)}})\\
&\leq & \frac{2c_1\max\{1,\tilde{c}\}}{R-\tilde{R}}\|1_{B_{R}}\|_{L^{\Phi(\cdot,\cdot)(\Omega)}}
\end{eqnarray*}
and hence
\begin{eqnarray*}
 R- \tilde{R} \leq {c_2} \frac{\|1_{B_R}\|_{L^{\Phi(\cdot,\cdot)(\Omega)}}}{ \| 1_{B_{\tilde{R}}}\|_{L^{\Psi(\cdot,\cdot)(\Omega)}}},
\end{eqnarray*}
where $c_2:=2c_1\max\{1,\tilde{c}\}.$\\

\textbf{Case-1} ($q(x) \geq 0 $ for all $x\in \Omega$): Using the norm estimates in Lemma \ref{Norm 1 lemma}, we get
 \begin{eqnarray*}
 R- \tilde{R} &\leq& c_2\ \frac{|A_R|^\frac{1}{p_{A_R}^+} (\log (e+\frac{1}{|A_R|}))^{q_{A_R}^+}} {|A_{\tilde{R}}|^\frac{1}{p*_{A_R^-}}} \\
              &=& c_2\ \frac{|A_R|^\frac{1}{p_{A_R}^+} (\log (e+\frac{1}{|A_R|}))^{q_{A_R}^+}} {|A_{\tilde R}|^{\frac{1}{p_{A_R}^-} -\frac{1}{n}}}\\
              &=&  c_2\ 2^{\frac{1}{p_{A_R}^-} -\frac{1}{n}}|A_R|^{\frac{1}{p_{A_R}^+} -\frac{1}{p_{A_R}^-} +\frac{1}{n}} (\log (e+\frac{1}{|A_R|}))^{q_{A_R}^+}\\
              &\leq & c_2\ 2^{\frac{1}{p^-} -\frac{1}{n}}\ |A_R|^{\frac{1}{p_{A_R}^+} -\frac{1}{p_{A_R}^-} +\frac{1}{n}} (\log (e+\frac{1}{|A_R|}))^{q_{A_R}^+},
\end{eqnarray*}
as claimed.\\

\textbf{Case-2} ($q(x) \leq 0$ for all $x\in \Omega$): Using the norm estimates in Lemma \ref{Norm 2 lemma}, we get 
\begin{eqnarray*}
 R- \tilde{R} &\leq& c_2\ \frac{|A_R|^\frac{1}{p_{A_R}^+}} {|A_{\tilde{R}}|^\frac{1}{p*_{A_R}^-} (\log (e+\frac{1}{|A_{\tilde{R}}|}))^\frac{q_{A_R}^- p*_{A_R}^+}{p*_{A_R}^- p_{A_R}^+ } } \\
              &=& c_2\ \frac{|A_R|^\frac{1}{p_{A_R}^+}} {|A_{\tilde R}|^{\frac{1}{p_{A_R}^-} -\frac{1}{n}}(\log (e+\frac{1}{|A_{\tilde{R}}|}))^{-Q_{A_R}}}\\
              &\equiv &  c_2\ 2^{\frac{1}{p_{A_R}^-} -\frac{1}{n}}|A_R|^{\frac{1}{p_{A_R}^+} -\frac{1}{p_{A_R}^-} +\frac{1}{n}} (\log (e+\frac{1}{|A_R|}))^{Q_{A_R}}\\
              &\leq & c_2\ 2^{\frac{1}{p^-} -\frac{1}{n}}\ |A_R|^{\frac{1}{p_{A_R}^+} -\frac{1}{p_{A_R}^-} +\frac{1}{n}} (\log (e+\frac{1}{|A_R|}))^{Q_{A_R}},
\end{eqnarray*}
where $Q_{A_R} =\frac{-q_{A_R}^- (n-p_{A_R}^-)}{p_{A_R}^- (n-p_{A_R}^+) } \geq 0. $\\
\textbf{Case-3} ($q(x) < 0$ for some $x\in \Omega$ and $q(x)\geq 0$ for some $x\in \Omega$): Using the norm estimates in Lemma \ref{Norm 3 lemma}, we get 
\begin{eqnarray*}
 R- \tilde{R} &\leq& c_2'\ \frac{|A_R|^\frac{1}{p_{A_R}^+}(\log (e+\frac{1}{|A_R|}))^{q_{A_R}^+}} {|A_{\tilde{R}}|^\frac{1}{p*_{A_R}^-} (\log (e+\frac{1}{|A_{\tilde{R}}|}))^\frac{q_{A_R}^- p*_{A_R}^+}{p*_{A_R}^- p_{A_R}^+ } } \\
              &=& c_2'\ \frac{|A_R|^\frac{1}{p_{A_R}^+}(\log (e+\frac{1}{|A_R|}))^{q_{A_R}^+}} {|A_{\tilde R}|^{\frac{1}{p_{A_R}^-} -\frac{1}{n}}(\log (e+\frac{1}{|A_{\tilde{R}}|}))^{-Q_{A_R}}}\\
              &\equiv &  c_2'\ 2^{\frac{1}{p_{A_R}^-} -\frac{1}{n}}|A_R|^{\frac{1}{p_{A_R}^+} -\frac{1}{p_{A_R}^-} +\frac{1}{n}} (\log (e+\frac{1}{|A_R|}))^{T_{A_R}}\\
              &\leq & c_2'\ 2^{\frac{1}{p^-} -\frac{1}{n}}\ |A_R|^{\frac{1}{p_{A_R}^+} -\frac{1}{p_{A_R}^-} +\frac{1}{n}} (\log (e+\frac{1}{|A_R|}))^{T_{A_R}}\\
              &\leq & c_2'\ 2^{\frac{1}{p^-} -\frac{1}{n}}\ |A_R|^{\frac{1}{p_{A_R}^+} -\frac{1}{p_{A_R}^-} +\frac{1}{n}} (\log (e+\frac{1}{|A_R|}))^{S_{A_R}},
\end{eqnarray*}
where $c_2'= \frac{c_2b_2}{b_1}$, $T_{A_R} = q_{A_R}^+ + Q_{A_R} = q_{A_R}^+ - \frac{q_{A_R}^-(n-p_{A_R}^-)}{p_{A_R}^- (n-p_{A_R}^+) } $ and $S_{A_R} = \max\{q_{A_R}^+,Q_{A_R},T_{A_R}\}.$
\end{proof}

\textbf{Proof of Theorem \ref{main theorem}} 
        For a fixed $x$ in $\bar\Omega$ define $A_R:= B_R(x)\cap \Omega$. It is enough to consider the case where $|A_R| \leq 1$; otherwise, $|A_R|\geq 1 \geq  R^n$ whenever $R\leq 1$, and there is nothing to prove. Moreover, it is enough to consider $R\leq r_0$ for some $0<r_0\leq 1/4.$ For such a $R$, denote by $\tilde{R}\leq R$ the smallest real number such that 
        \begin{center}
            $|A_{\Tilde{R}}|= \frac{1}{2} |A_R|$
        \end{center}

Now, construct the sequence $\{  R_i \}$ by setting $R_0:=R$, and then define  $R_{i+1}:=\tilde{R_i}$ inductively for $i \geq 0$.
It follows that
\[
     |A_{R_{i}}|=\frac{1}{2^i}\ | A_R|, 
\]
with $\lim_{i \rightarrow \infty} R_i=0$.\\

\textbf{Case-1} ($q(x) \geq 0 $ for all $x\in \Omega$): Using Lemma \ref{main lemma } one obtains
\begin{eqnarray*}     
  R_i -R_{i+1} &\leq& C_1 |A_{R_i}|^{\frac{1}{n}+\frac{1}{p_{A_{R_i}}^+}-\frac{1}{p_{A_{R_i}}^-}}(\log (e+\frac{1}{|A_{R_i}|}))^{q_{A_{R_i}}^+}\\ 
	      &\leq&  C_1 |A_{R_i}|^{\frac{1}{n}+\frac{1}{p_{A_R}^+}-\frac{1}{p_{A_R}^-}}(\log (e+\frac{1}{|A_{R_i}|}))^{q_{A_R}^+}\\
             &=& C_1 \frac{|A_R|^{\eta_R}}{2^{i \eta_R}}(\log (e+\frac{2^i}{|A_{R}|}))^{q_{A_R}^+},
\end{eqnarray*}
where $\eta_R :=\frac{1}{n}+\frac{1}{p_{A_R}^+}-\frac{1}{p_{A_R}^-}.$  
Since we have $(\log (e+\frac{2^i}{|A_{R}|}))^{q_{A_R}^+}\leq i^{q_{A_R}^+}(\log (e+\frac{2}{|A_{R}|}))^{q_{A_R}^+}\leq 2^{q^+}i^{q_{A_R}^+}(\log (e+\frac{1}{|A_{R}|}))^{q_{A_R}^+}$ for $i\geq 1,$ we obtain 
\begin{eqnarray*}     
  R_i -R_{i+1} 
             &\leq & c_3\frac{i^{q_{A_R}^+}|A_R|^{\eta_R}}{2^{i \eta_R}}(\log (e+\frac{1}{|A_{R}|}))^{q_{A_R}^+},
\end{eqnarray*}
where $c_3={C_1}2^{q^+}.$\\

Now, we would like to find a constant $\eta>0$, independent of $x$ and $R$, such that $ \eta_R\geq \eta>0$ for all $R\leq r_0.$ Toward this end, the log-H\"older continuity of $p$ gives, for any $z$ and $y$ in $A_R,$
\[
\vert \frac{1}{p(z)}-\frac{1}{p(y)}\vert\leq \frac{C_{\text{log}}}{\log(e+1/|z-y|)},
\]
and taking the supremum over all pairs of points in $A_R$ one gets
\begin{equation}   \label{estimate1/p}
\frac{1}{p_{A_R}^-}-\frac{1}{p_{A_R}^+}\leq \frac{C_{\text{log}}}{\log(1/(2R))}.
\end{equation}
Suppose now that for some $R\leq 1/4$ we have $\eta_R\leq 0$. Then (\ref{estimate1/p}) gives
$ \frac{1}{n}\leq \frac{C_{\text{log}}}{\log\left(\frac{1}{2R}\right)}$,
which further implies
$R\geq \frac{1}{2}e^{-nC_{\text{log}}}$.

Hence, we have the following conclusion:
\begin{itemize} 
\item If $\frac{1}{2}e^{-nC_{\text{log}}}>\frac{1}{4}$, then there is no $R\leq \frac{1}{4}$ for which $\eta_R\leq 0$.
\item If $\frac{1}{2}e^{-nC_{\text{log}}}\leq \frac{1}{4}$, then $\eta_R\leq 0$ implies $R\geq \frac{1}{2}e^{-nC_{\text{log}}}.$ 
\end{itemize}

Therefore, if we choose $r_0=\frac{1}{2}\min\{\frac{1}{4},\frac{1}{2}e^{-nC_{\text{log}}}\}$, then $\eta_{r_0}>0$, and also
\begin{equation}\label{estimate of r_0} 
 \frac{1}{n}>\frac{C_{\text{log}}}{\log(1/(2r_0))}.
\end{equation} 
  
But $\eta_{r_0}$ may depend on the point $x$ fixed at the beginning of the proof. To obtain the required $\tilde{\eta}$, we apply again log-H\"older continuity of $1/p$ on $A_{r_0}$, to obtain
 \begin{equation}\label{estimate in A_r_0}
 \frac{1}{p_{A_{r_0}}^-}-\frac{1}{p_{A_{r_0}}^+}\leq \frac{C_{\text{log}}}{\log(1/(2r_0))},
 \end{equation}
 and \eqref{estimate of r_0} together with \eqref{estimate in A_r_0} give
 \[
 \eta_{r_0}=\frac{1}{n}+\frac{1}{p_{A_{r_0}}^+}-\frac{1}{p_{A_{r_0}}^-}\geq \frac{1}{n}-\frac{C_{\text{log}}}{\log(1/(2r_0))}>0.
 \]
 Choosing $\eta:=\frac{1}{n}-\frac{C_{\text{log}}}{\log(1/(2r_0))},$ we get that $\eta_R\geq \eta_{r_0}\geq \eta>0$ for all $R\leq r_0$. This is our desired $\eta.$\\

Note that $\eta_R \geq \eta >0$ and by the integral test $\sum_{i=1}^{\infty} {i^{q_{A_R}^+}}{2^{-i\eta_R}}\leq \frac{(q_{A_R}^+)!}{(\eta_R \ln2)^{(q_{A_R}^+ +1)}}.$ Here for any positive real number $a$ we denote $a! = \Gamma(a+1),$ where $\Gamma$ denotes the Gamma function. Hence we have
\begin{eqnarray*}     
 R = \sum_{i=0}^{\infty }(R_i -R_{i+1}) 
 &\leq& c_3 |A_R|^{\eta_R}(\log (e+\frac{1}{|A_{R}|}))^{q_{A_R}^+} \Big(1+\sum_{i=1}^{\infty} {i^{q_{A_R}^+}}{2^{-i\eta_R}}\Big)\\
  &\leq& c_3 |A_R|^{\eta_R}(\log (e+\frac{1}{|A_{R}|}))^{q_{A_R}^+}\Big(\frac{(q_{A_R}^+)!}{(\eta_R \ln2)^{(q_{A_R}^+ +1)}}+1\Big)\\
   &\leq& c_3 |A_R|^{\eta_R}(\log (e+\frac{1}{|A_{R}|}))^{q_{A_R}^+}\Big(\frac{(q_{A_R}^+)!}{(\eta \ln2)^{(q_{A_R}^+ +1)}}+1\Big)\\
   &\leq& c_3 |A_R|^{\eta_R}(\log (e+\frac{1}{|A_{R}|}))^{q_{A_R}^+}\Big(\frac{(q^+)!}{({\min(1,\eta)} \ln2)^{(q^+ +1)}}+1\Big).
 \end{eqnarray*}
Moreover, since $ c_4:=1/ \max\Big\{1,\Big(\frac{c_3(q^+)!}{({\min(1,\eta)} \ln2)^{(q^+ +1)}}+c_3\Big)\Big\} \leq 1$ one has 
\begin{eqnarray}    \label{AR}
 |A_R| (\log (e+\frac{1}{|A_{R}|}))^\frac{q_{A_R}^+}{\eta_R}  \geq   c_4^{1 /\eta_R}  R^{1/\eta_R}
 \geq  c_4^{ 1 /\eta} R^{1 /\eta_R}
 = c_4^{1 / \eta} R^n R^{\beta_R / \eta_R} ,
\end{eqnarray}
where $\beta_R := 1-n \eta_R$.
Therefore, from equation \eqref{AR} one sees that if a positive lower bound for $R^{\beta_R / \eta_R}$ 
is provided, the proof of Theorem \ref{main theorem}
is finished. To achieve such a lower bound, we see that from the log-H\"{o}lder continuity of $p,$ 
\[
\vert p(z)-p(y)\vert\leq \frac{C_{\text{log}}}{\log(e+1/|z-y|)} ;
\]
taking the supremum over pairs of points in $A_R$, one gets
\[
p_{A_R}^+-p_{A_R}^-\leq \frac{C_{\text{log}}}{\log(1/(2R))} ,
\]
or
\[
\log\left(1 / (2R)^{p_{A_R}^+-p_{A_R}^-}\right)\leq C_{\text{log}},
\]
therefore,
\begin{equation} \label{inequality 4}    
  R^{p_{A_R}^+-p_{A_R}^-}\geq \frac{e^{-C_{\text{log}}}}{2^{p_{A_R}^+- p_{A_R}^-  }} \geq \frac{e^{-C_{\text{log}}}}{2^{(p^+-p^-)}}.
\end{equation}

But
\[
R^{\frac{\beta_R}{\eta_R}} \geq R^{\frac{\beta_R}{\eta}}
= R^{\frac{n(p_{A_R}^+ - p_{A_R}^-)}{\eta p_{A_R}^+  p_{A_R}^-  }} \geq   (R^{p_{A_R}^+ - p_{A_R}^-})^{n / \eta (p^-)^2} ,
\]
hence using \eqref{inequality 4} the required bound
\[  
R^{\frac{\beta_R}{\eta_R}} \geq  \left(\frac{e^{-C_{\text{log}}}}{2^{(p^+-p^-)}} \right)^{n / \eta (p^-)^2} =: c_5 >0.
\]
 Taking $f(t)=t (\log (e+\frac{1}{t}))^\frac{q_{A_R}^+}{\eta_R}$, we see that \eqref{AR} becomes $f(|A_R|)\geq cR^n,$ where $c := c_4^{1 / \eta} c_5$ and hence $|A_R|\geq  f^{-1}(cR^n)$ which further implies that $c
 R^n (\log (e+\frac{1}{R}))^\frac{-q^+}{\eta} \leq c
 R^n (\log (e+\frac{1}{R}))^\frac{-q_{A_R}^+}{\eta_R} \leq | B_R(x) \cap \Omega | $. Hence $\Omega$ satisfies the log-measure density condition.\\
 
\textbf{Case-2} ($q(x) \leq 0$ for all $x\in \Omega$): Using Lemma \ref{main lemma } one obtains 
\begin{eqnarray*}     
  R_i -R_{i+1} &\leq& C_2 |A_{R_i}|^{\frac{1}{n}+\frac{1}{p_{A_{R_i}}^+}-\frac{1}{p_{A_{R_i}}^-}}(\log (e+\frac{1}{|A_{R_i}|}))^{Q_{A_R}}\\ 
	      &\leq&  C_2 |A_{R_i}|^{\frac{1}{n}+\frac{1}{p_{A_R}^+}-\frac{1}{p_{A_R}^-}}(\log (e+\frac{1}{|A_{R}|}))^{Q_{A_R}}\\
             &=& C_2 \frac{|A_R|^{\eta_R}}{2^{i \eta_R}}(\log (e+\frac{2^i}{|A_{R}|}))^{Q_{A_R}} \\
             &\leq & c_6\frac{2^{Q}i^{Q_{A_R}}|A_R|^{\eta_R}}{2^{i \eta_R}}(\log (e+\frac{1}{|A_{R}|}))^{Q_{A_R}},
\end{eqnarray*}
where $\eta_R :=\frac{1}{n}+\frac{1}{p_{A_R}^+}-\frac{1}{p_{A_R}^-}$ and $c_6=C_2(2)^{Q} .$ Note that $\eta_R \geq \eta >0$ and hence
\begin{eqnarray*}     
 R = \sum_{i=0}^{\infty }(R_i -R_{i+1}) 
 &\leq& c_6 |A_R|^{\eta_R}(\log (e+\frac{1}{|A_{R}|}))^{Q_{A_R}} \Big(1+\sum_{i=1}^{\infty} {i^{Q_{A_R}}}{2^{-i\eta_R}}\Big)\\
  &\leq& c_6 |A_R|^{\eta_R}(\log (e+\frac{1}{|A_{R}|}))^{Q_{A_R}}\Big(1+\frac{(Q_{A_R})!}{(\eta_R \ln2)^{({Q_{A_R}} +1)}}\Big)\\
   &\leq& c_6 |A_R|^{\eta_R}(\log (e+\frac{1}{|A_{R}|}))^{Q_{A_R}} \Big(1+\frac{(Q_{A_R} )!}{(\eta \ln2)^{(Q_{A_R} +1)}}\Big)\\
   &\leq& c_6 |A_R|^{\eta_R}(\log (e+\frac{1}{|A_{R}|}))^{Q_{A_R}} \Big(1+\frac{(Q )!}{(\min(1,\eta) \ln2)^{(Q +1)}}\Big),
 \end{eqnarray*}
 where $Q =  \frac{-q^-(n-p^-)}{p^-(n-p^+) }\geq Q_{A_R} \geq 0.$
 Moreover, since $ c_7:=1/ \max\Big\{\Big(1,c_6+\frac{c_6(Q )!}{(\min(1,\eta) \ln2)^{(Q +1)}}\Big)\Big\} \leq 1,$ one has
 \begin{eqnarray}    \label{BR}
 |A_R| (\log (e+\frac{1}{|A_{R}|}))^\frac{Q_{A_R}}{\eta_R}  \geq   c_4^{1 /\eta_R}  R^{1/\eta_R}
 \geq  c_7^{ 1 /\eta} R^{1 /\eta_R}
 = c_7^{1 / \eta} R^n R^{\beta_R / \eta_R} ,
\end{eqnarray}
where $\beta_R := 1-n \eta_R$.
Now we can proceed in a similar way as in Case-1 to obtain $c
 R^n (\log (e+\frac{1}{R}))^\frac{-Q}{\eta} \leq c
 R^n (\log (e+\frac{1}{R}))^\frac{-Q_{A_R}}{\eta_R} \leq | B_R(x) \cap \Omega | $. Hence $\Omega$ satisfies the log-measure density condition.\\
 
 \textbf{Case-3} ($q(x) < 0$ for some $x\in \Omega$ and $q(x)\geq 0$ for some $x\in \Omega$): Using the Lemma \ref{main lemma } one obtains 
\begin{eqnarray*}     
  R_i -R_{i+1} &\leq& C_3 |A_{R_i}|^{\frac{1}{n}+\frac{1}{p_{A_{R_i}}^+}-\frac{1}{p_{A_{R_i}}^-}}(\log (e+\frac{1}{|A_{R_i}|}))^{S_{A_R}}\\ 
	      &\leq&  C_3 |A_{R_i}|^{\frac{1}{n}+\frac{1}{p_{A_R}^+}-\frac{1}{p_{A_R}^-}}(\log (e+\frac{1}{|A_{R}|}))^{S_{A_R}}\\
             &=& C_3 \frac{|A_R|^{\eta_R}}{2^{i \eta_R}}(\log (e+\frac{2^i}{|A_{R}|}))^{S_{A_R}} \\
             &\leq & c_8\frac{2^{S}i^{S_{A_R}}|A_R|^{\eta_R}}{2^{i \eta_R}}(\log (e+\frac{1}{|A_{R}|}))^{S_{A_R}},
\end{eqnarray*}
where $\eta_R :=\frac{1}{n}+\frac{1}{p_{A_R}^+}-\frac{1}{p_{A_R}^-}$ and $c_8=C_3(2)^{S} .$ Note that $\eta_R \geq \eta >0$ and hence
\begin{eqnarray*}     
 R = \sum_{i=0}^{\infty }(R_i -R_{i+1}) 
 &\leq& c_8 |A_R|^{\eta_R}(\log (e+\frac{1}{|A_{R}|}))^{S_{A_R}} \Big(1+\sum_{i=1}^{\infty} {i^{S_{A_R}}}{2^{-i\eta_R}}\Big)\\
  &\leq& c_8 |A_R|^{\eta_R}(\log (e+\frac{1}{|A_{R}|}))^{S_{A_R}}\Big(1+\frac{(S_{A_R}!)}{(\eta_R \ln2)^{({S_{A_R}} +1)}}\Big)\\
   &\leq& c_8 |A_R|^{\eta_R}(\log (e+\frac{1}{|A_{R}|}))^{S_{A_R}} \Big(1+\frac{(S_{A_R} )!}{(\eta \ln2)^{(S_{A_R} +1)}}\Big)\\
   &\leq& c_8 |A_R|^{\eta_R}(\log (e+\frac{1}{|A_{R}|}))^{S_{A_R}} \Big(1+\frac{(S )!}{(\min(1,\eta) \ln2)^{(S +1)}}\Big),
 \end{eqnarray*}
 where $S= \max\{q^+,Q,T\} \geq S_{A_R}\geq 0$ and $T = q^+ -\frac{q^-(n-p^-)}{p^-(n-p^+) }\geq T_{A_R}.$
 Moreover, since $ c_9:=1/ \max\Big \{\Big(1,c_8+\frac{c_8(S )!}{(\min(1,\eta) \ln2)^{(S +1)}}\Big)\Big\} \leq 1$ one has
 \begin{eqnarray}    \label{BR2}
 |A_R| (\log (e+\frac{1}{|A_{R}|}))^\frac{S_{A_R}}{\eta_R}  \geq   c_9^{1 /\eta_R}  R^{1/\eta_R}
 \geq  c_9^{ 1 /\eta} R^{1 /\eta_R}
 = c_9^{1 / \eta} R^n R^{\beta_R / \eta_R} ,
\end{eqnarray}
where $\beta_R := 1-n \eta_R$.
Now we can proceed in a similar way as in Case-1 to obtain $c
 R^n (\log (e+\frac{1}{R}))^\frac{-S}{\eta} \leq c
 R^n (\log (e+\frac{1}{R}))^\frac{-S_{A_R}}{\eta_R} \leq | B_R(x) \cap \Omega | $. Hence $\Omega$ satisfies the log-measure density condition. \qed \\

\bigskip
\noindent\textbf{Conflict of interest:} The authors declare no conflict of interest.\\
\textbf{Data Availability Statement:} Data sharing is not applicable to this article, as no data sets were
generated or analyzed during the current study.\\
\noindent\textbf{Contributions:} All authors have contributed equally to the manuscript.

\def\bibname{References}
\bibliography{musielak}
\bibliographystyle{acm}

\bigskip

\noindent{\small Ankur Pandey}\\
\small{Department of Mathematics,}\\
\small{Birla Institute of Technology and Science-Pilani, Hyderabad Campus,}\\
\small{Hyderabad-500078, India} \\
{\tt p20210424@hyderabad.bits-pilani.ac.in; pandeyankur600@gmail.com}\\

\noindent{\small Nijjwal Karak}\\
\small{Department of Mathematics,}\\
\small{Birla Institute of Technology and Science-Pilani, Hyderabad Campus,}\\
\small{Hyderabad-500078, India} \\
{\tt nijjwal@gmail.com ; nijjwal@hyderabad.bits-pilani.ac.in}\\
\\

\end{document}